\documentclass[11pt]{amsart}

\usepackage[T1]{fontenc}
\usepackage[latin1]{inputenc}
\usepackage[english]{babel}
\usepackage{amsmath,amssymb,amsthm} 
\usepackage{enumitem}
\usepackage{mymacros}
\newcommand{\CRC}{\overline{\mathfrak{C}\mathfrak{B}_d(n)}}
\newcommand{\CRCo}{\mathfrak{C}\mathfrak{B}_d(n)}
\usepackage[margin=1.5in]{geometry}
\usepackage{color}
\usepackage{hyperref}

\title{von Neumann's inequality for row contractive matrix tuples}
\author{Michael Hartz}
\address{Fachrichtung Mathematik, Universit\"at des Saarlandes, 66123 Saarbr\"ucken, Germany}
\email{hartz@math.uni-sb.de}
\author{Stefan Richter}
\address{Department of Mathematics, University of Tennessee, 1403 Circle Drive, Knoxville, TN 37996-1320, USA}
\email{srichter@utk.edu}
\author{Orr Moshe Shalit}
\address{Faculty of Mathematics\\
Technion - Israel Institute of Technology\\
Haifa\; 3200003\\
Israel}
\email{oshalit@technion.ac.il \vspace{-2ex}}
{}
\thanks{The work of M.H.\ is partially supported by a GIF grant. The work of O.M. Shalit is partially supported by ISF Grants no. 195/16 and 431/20.}

\date{\today}
\address{}
\email{}


\begin{document}

\begin{abstract}
We prove that for all $n\in \bN$, there exists a constant $C_{n}$ such that for all $d \in \bN$, for every row contraction $T$ consisting of $d$ commuting $n \times n$ matrices and every polynomial $p$, the following inequality holds:
\[
 \|p(T)\| \le C_{n} \sup_{z \in \bB_d} |p(z)| .
\]
We apply this result and the considerations involved in the proof to several open problems from the pertinent literature. 
First, we show that Gleason's problem cannot be solved contractively in $H^\infty(\mathbb{B}_d)$ for $d \ge 2$.
Second, we prove that the multiplier algebra $\Mult(\mathcal{D}_a(\mathbb{B}_d))$ of the weighted Dirichlet space $\mathcal{D}_a(\mathbb{B}_d)$ on the ball is not topologically subhomogeneous when $d \ge 2$ and $a \in (0,d)$. 
In fact, we determine all the bounded finite dimensional representations of the norm closed subalgebra $A(\mathcal{D}_a(\mathbb{B}_d))$ of $\Mult(\mathcal{D}_a(\mathbb{B}_d))$ generated by polynomials. 
Lastly, we also show that there exists a uniformly bounded nc holomorphic function on the free commutative ball $\mathfrak{C}\mathfrak{B}_d$ that is levelwise uniformly continuous but not globally uniformly continuous.
\end{abstract}

\maketitle

\section{Introduction}
Recall that a tuple $T = (T_1, \ldots, T_d) \in B(H)^d$ of operators on a Hilbert space $H$ is said to be a row contraction if $\|T\| = \|\sum_{i=1}^d T_i T_i^*\|^{1/2} \leq 1$. 
We say that $T$ is a strict row contraction if $\|T\|<1$. 
If in addition $T_i T_j  = T_jT_i $ for all $i,j$ then we say that $T$ is a commuting row contraction. 

The central result of this paper is the following theorem, which answers Question 9.15 in \cite{SalShaSha18}.

\begin{thm}
  \label{thm:main}
  Let $n \in \mathbb{N}$.
  There exists a constant $C_{n}$ such that for all $d \in \mathbb{N}$,
  for every commuting row contraction $T=(T_1,\ldots,T_d)$ on a Hilbert space of dimension $n$ and for every
  polynomial $p \in \mathbb{C}[z_1,\ldots,z_d]$, the inequality
  \begin{equation*}
    \|p(T)\| \le C_{n} \sup_{z \in \mathbb{B}_d} |p(z)|
  \end{equation*}
  holds.
\end{thm}

This result is the content of Theorem \ref{thm:vNrow} below.
For each $d,n \in \mathbb{N}$, we find an explicit upper bound for constants $C_{d,n}$ such that
\begin{equation}\label{eq:ineq}
  \|p(T)\| \le C_{d,n} \sup_{z \in \mathbb{B}_d} |p(z)|
\end{equation}
holds for every row contraction $T$ consisting of $d$ commuting $n \times n$ matrices. 
The inequality \eqref{eq:ineq} with a constant that possibly depends on $d$ is the main result of this paper. 
By applying essentially linear algebraic considerations, we will show that the best constants
$C_{d,n}$ are bounded in $d$ for fixed $n$, so that we may define $C_n = \sup_d C_{d,n}$.
However, for fixed $d \in \mathbb{N}$, the constants $C_{d,n}$
converge to infinity at a slower rate than the constants $C_{n}$.
Although it is very likely that the expression we find for general $d,n$ is a gross overestimation of the optimal constant, for the special case $n = 2$ we find the best possible constant; in fact, we show that $C_{d,2} = 1$ for all $d$; see Corollary \ref{cor:two}. 

It is well known that when $d \geq 2$ no constant can be found that will satisfy the inequality \eqref{eq:ineq} for all $n$.
To see this, recall that the supremum of $\|p(T)\|$ as $T$ ranges over all $n \times n$ row contractions and all $n$ is equal to the multiplier norm $\|p\|_{\Mult(H^2_d)}$ of $p$, considered as a multiplier on the Drury-Arveson space $H^2_d$ (see, for example, \cite[Section 11]{SalShaSha18}).
On the other hand, we have on the right hand side of the inequality the supremum norm $\|p\|_\infty := \sup_{z \in \mathbb{B}_d} |p(z)|$ of $p$ on the Euclidean unit ball $\mathbb{B}_d$ in $\mathbb{C}^d$.
The incomparability of the multiplier norm and the supremum norm was already observed by Drury \cite{Dru78}. 

The incomparability of the multiplier and supremum norms notwithstanding, one might guess that Theorem \ref{thm:main}, at least in the form of inequality \eqref{eq:ineq}, can be obtained by a straightforward application of standard techniques, since the row contractions appearing in it are restricted to act on spaces of a fixed finite dimension. 
However, we found that some new ideas are needed in order to prove the existence of the constants $C_{d,n}$.
It is worth highlighting that as a consequence of Theorem \ref{thm:main} and of the techniques used in the proof,
we obtain several results that answer other questions in the literature.
We now survey these additional results.

To simplify notation, let $\CRCo \subseteq M_n(\bC)^d$ denote the set of all strict row contractions consisting of $d$ commuting $n \times n$ matrices, and let $\CRC$ denote its closure, the set of all commuting $n \times n$ row contractions. The set
\begin{equation*}
  \mathfrak{C}\mathfrak{B}_d = \bigsqcup_{n=1}^\infty \CRCo
\end{equation*}
is called the free commutative ball.

In Section \ref{sec:small}, we make a connection between the $n$-point muliplier norm $\|f\|_{\Mult(H^2_d),n}$, defined in \cite{AHM+20a}, and Theorem \ref{thm:main}.
In Proposition \ref{prop:diagonalizable} we prove that
\begin{equation*}
\|f\|_{\Mult(H^2_d),n} = \sup \{ \|f(T)\| : T \in \CRCo \text{ diagonalizable} \}.
\end{equation*}
This is used to prove that $C_{d,2} = 1$ for all $d$. 

It is natural to wonder whether $C_{d,n} = 1$ for other values of $d,n$ (besides the well known $C_{1,n} = 1$, which is just von Neumann's inequality). 
We answer this in Proposition \ref{prop:3_2}, which shows that $C_{2,3} > 1$ and hence $C_{d,n} > 1$ whenever $d\geq 2$ and $n \geq 3$. This result is used in Section \ref{sec:Gleason} to show that Gleason's problem cannot be solved contractively in $H^\infty(\bB_d)$ for $d \geq 2$.

For $a > 0$, let $\mathcal{D}_a(\mathbb{B}_d)$ be the reproducing kernel Hilbert space (RKHS for short) on $\mathbb{B}_d$ with reproducing kernel 
\[
k(z,w) = \frac{1}{(1 - \langle z,w \rangle)^a} .
\]
If $a=1$, then $\mathcal{D}_a(\mathbb{B}_d) = H^2_d$, the Drury--Arveson space.
Section \ref{sec:VNI} is concerned with the relationship between operator theory and the multiplier algebra of $\mathcal{D}_a(\mathbb{B}_d)$. 
The main technical achievement is Lemma \ref{lem:Schur_improved}, in which we show that if $f \in H^\infty(\mathbb{B}_d)$ and $T$ is a tuple of commuting $n \times n$ matrices whose joint
  spectrum is contained in $\mathbb{B}_d$, then there exists a function $g \in \Mult(\mathcal{D}_a(\mathbb{B}_d))$ with $f(T) = g(T)$ and $\|g\|_{\Mult(\mathcal{D}_a(\mathbb{B}_d))} \le C \|f\|_\infty$, where $C$ is a constant that depends only on $n$ and $a$. 
Since the multiplier norms are well behaved with respect to the holomorphic functional calculus, this lemma allows us to control the norm of $f(T)$ in terms of $\|f\|_\infty$. 
An immediate consequence of this result is Theorem \ref{thm:vNrow}, which is a refined version of Theorem \ref{thm:main}.

In Section \ref{sec:reps} we employ the tools from Section \ref{sec:VNI} to study the representation theory of the multiplier algebras $\Mult(\mathcal{D}_a(\mathbb{B}_d))$ and of their norm closed subalgebras $A(\mathcal{D}_a(\mathbb{B}_d))$ generated by the polynomials.
We give a complete description of the bounded finite dimensional representations of $A(\mathcal{D}_a(\mathbb{B}_d))$, and we also show that the algebras $A(\mathcal{D}_a(\mathbb{B}_d))$, and hence $\Mult(\mathcal{D}_a(\mathbb{B}_d))$, are not topologically subhomogeneous for $a<0<d$, thereby solving an open problem from \cite{AHM+20a}. 

We conclude this paper by solving an open problem from \cite{SalShaSha18}: we show, in Proposition \ref{prop:lw_uniform}, that there exists a function $f \in \Mult(H^2_d)$ that gives rise to a noncommutative function on the closed free commutative unit ball $\sqcup_{n=1}^\infty \CRC$ that is levelwise uniformly continuous, but not globally uniformly continuous (see Section \ref{sec:application} for details).
It might be interesting to note that the question behind Theorem \ref{thm:main} grew out of an earlier attempt to settle this problem on uniform continuity.

\vskip 10pt

\noindent{\bf Acknowledgements.} The collaboration leading to this paper was spurred by the presentation of the question behind Theorem \ref{thm:main} by one of the authors in the ``Open Problems Session'' that was held at the online conference OTWIA 2020.
We wish to thank Meric Augat, the organizer of that session, as well as the organizers of the conference. 

We are also grateful to \L ukasz Kosi\'{n}ski for helpful comments and for bringing \cite{KZ18} to our attention.
Moreover, we are grateful to the editor Mikael de la Salle for bringing to our attention a theorem of Schur, as well \cite{Mir98}, where an elementary proof can be found. This led to an improved upper bound in Lemma \ref{lem:variable_reduction}.

\section{Preliminaries on multvariable spectral theory}\label{sec:prelim}

We will require some elementary facts about the spectrum of a commuting tuple of matrices, which we now briefly review. 
For more information on joint spectra see the monograph \cite{Mueller07}. 

Recall that given a commutative unital Banach algebra $\cB$ with maximal ideal space $\Delta(\cB)$, and a $d$-tuple $a = (a_1, \ldots, a_d)\in\cB^d$, the joint spectrum of $a$ with respect to $\cB$ is the subset of $\bC^d$ defined by  
\[
\sigma_{\cB}(a) = \{(\chi(a_1), \ldots, \chi(a_d)) : \chi \in \Delta(\cB)\}.
\]
When the algebra $\cB$ is understood we simply write $\sigma(a)$.
Sometimes the joint spectrum is referred to simply as spectrum. 
This is perhaps the simplest notion of spectrum and it will suffice for our needs. 

If $T = (T_1, \ldots, T_d) \in B(H)^d$ is a tuple of commuting operators on a Hilbert space, then there is also the notion of Taylor spectrum. 
We shall not define the Taylor spectrum, but we remark that it is contained in $\sigma_{\cB}(T)$ for any commutative unital Banach algebra $\cB \subseteq B(H)$ that contains $T_1, \ldots, T_d$.
In any case, when $H$ is finite dimensional then the spectrum $\sigma(T) = \sigma_\cB(T)$ is independent of the unital commutative algebra $\cB$ that contains $T$, and is given as the set of points 
\[
\sigma(T) = \Big\{\big(\langle T_1v_i, v_i \rangle, \ldots, \langle T_dv_i, v_i \rangle\big) : i=1, \ldots, n\Big\} \subset \bC^d,
\]
where $v_1, \ldots, v_n$ is an orthonormal basis for $H$ in which $T_1, \ldots, T_d$ are jointly upper triangular. 
The above set is also equal to the Taylor spectrum as well as to the so-called Waelbroeck spectrum. The equality of all these spectra in the finite dimensional seetting is explained nicely in Section 2.1 in \cite{Coh15}.

If $T$ is a commuting row contraction,
then the joint spectrum of $T$ with respect to the unital Banach algebra generated by $T$ is contained in
the closed unit ball $\overline{\mathbb{B}}_d$. Indeed, this follows from the fact that characters
on operator algebras are automatically completely contractive.

We will also require a basic holomorphic functional calculus for commuting tuples of matrices,
which can be regarded as a very special case of the Arens--Calderon functional calculus or of the Taylor functional calculus,
see \cite[Section 30]{Mueller07}.
Explicitly, we will use that if $T$ is a tuple of commuting matrices
whose joint spectrum is contained in $\mathbb{B}_d$, then
the ordinary polynomial functional calculus
$p \mapsto p(T)$  extends to a continuous algebra homomorphism on the algebra $\mathcal{O}(\mathbb{B}_d)$ of all holomorphic functions
on $\mathbb{B}_d$; we denote the extended homomorphism by $f \mapsto f(T)$.
If $T$ is jointly diagonalizable, then $f(T)$ can simply be computed by applying $f$ to the diagonal entries of a diagonal representation of $T$.
As the general constructions of the Arens--Calderon and of the Taylor functional calculus are somewhat involved, we provide
an elementary construction that is sufficient for our needs in Theorem \ref{thm:fc_elementary} in the appendix.

\section{Small matrices and the \texorpdfstring{$n$}{n}-point norm}\label{sec:small}

First, we observe that the question behind Theorem \ref{thm:main} is closely related to the relationship between the $n$-point
multiplier norm on the Drury--Arveson space and the sup norm.
To recall the definition of the $n$-point multiplier norm, let $\mathcal{H}$
be a reproducing kernel Hilbert space of functions on $\mathbb{B}_d$.
For background on reproducing kernel Hilbert spaces, see \cite{AM02,PR16}.
For $F \subset \mathbb{B}_d$, we denote by $\mathcal{H} \big|_F$
the reproducing kernel Hilbert space on $F$ whose
reproducing kernel is the restriction of the reproducing kernel of $\mathcal{H}$ to $F \times F$.
For $n \in \mathbb{N}$ with $n \ge 1$, the \emph{$n$-point multiplier norm} of a function
$f: {\mathbb{B}_d} \to \mathbb{C}$ is defined as
\begin{equation*}
  \|f\|_{\Mult(\mathcal{H}),n} = \sup \left\{ \|f \big|_F \|_{\Mult(\mathcal{H} \big|_F)} : F \subset \mathbb{B}_d \text{ with } |F| \le n \right\}.
\end{equation*}
Clearly, the condition $|F| \le n$ can be replaced with $|F| = n$.
See \cite{AHM+20a} for background on the $n$-point muliplier norm.

\begin{prop}\label{prop:diagonalizable}
  \label{prop:n_point_diagonalizable}
Let $f$ be holomorphic in a neighborhood of $\overline{\mathbb{B}}_d$. Then
  \begin{equation*}
  \|f\|_{\Mult(H^2_d),n} = \sup \{ \|f(T)\| : T \in \CRC \text{ diagonalizable} \}.
  \end{equation*}
\end{prop}

\begin{proof}
  We first show the inequality ``$\le$''.
  Given a subset $F \subset \mathbb{B}_d$ with $|F|=n$, consider the $n$-dimensional RKHS $H^2_d \big|_F$ on $F$
  and the tuple $(T_1,\ldots,T_d)$, where $T_i$ is the operator of multiplication by $z_i$ on $H^2_d \big|_F$.
  Clearly, $T$ is a commuting row contraction. Moreover, the tuple $T$ is jointly diagonalizable, as the kernel
  functions at the points in $F$ form a basis of $H^2_d \big|_F$ and are joint eigenvectors of the adjoint tuple $T^*$.
  Moreover,
  \begin{equation*}
    \|f\|_{\Mult(H^2_d|_F)} = \|f(T)\|.
  \end{equation*}
  Taking the supremum over all subsets $F$ of $\mathbb{B}_d$ with $|F| = n$, we therefore find that
  \begin{equation*}
    \|f\|_{\Mult(H^2_d),n} \le \sup \{ \|f(T)\| : T \in \CRC \text{ diagonalizable} \}.
  \end{equation*}

  Conversely, let $T \in \CRC$ be diagonalizable and suppose that $f$ is holomorphic with $\|f\|_{\Mult(H^2_d),n} \le 1$.
  We wish to show that $\|f(T)\| \le 1$. 
  An approximation argument shows that we may assume that $T \in \CRCo$.
  Let $\lambda_1,\ldots,\lambda_n \in {\mathbb{B}_d}$ be the joint eigenvalues of $T$
  and set $F = \{\lambda_1,\ldots,\lambda_n\}$.
  Since $T$ is diagonalizable, if $g$ is another holomorphic function on $\mathbb{B}_d$ that agrees with $f$ on $F$,
  then $f(T) = g(T)$.
  Now, since $\|f\|_{\Mult(H^2_d),n} \le 1$, we have $\|f \big|_F \|_{\Mult(H^2_d \big|_F)} \le 1$, so by the Pick property
  of $H^2_d$ \cite[Chapter 8]{AM02}, there exists $g \in \Mult(H^2_d)$ with $g \big|_F = f \big|_F$ and $\|g\|_{\Mult(H^2_d)} \le 1$.
  Consequently, by the von Neumann inequality for $H^2_d$ of Drury \cite{Dru78}, M\"uller-Vasilescu \cite{MV93} and Arveson \cite{Arveson98}, we find that
  \begin{equation*}
    \|f(T)\| = \|g(T)\| \le 1,
  \end{equation*}
  as desired.
\end{proof}

If $f \in \Mult(H^2_d)$, then for all $r \in (0,1)$ the function $f_r(z) := f(rz)$ is holomorphic in a neighborhood of $\ol{\bB_d}$, and 
\[
\sup_{0<r<1} \|f_r\|_{\Mult(H^2_d)} = \|f\|_{\Mult(H^2_d)},
\] 
see, e.g.\ \cite[Theorem 3.5.5]{ShalitSurvey}.
Since $\|f\|_{\Mult(H^2_d)}  = \sup_{n} \|f\|_{\Mult(H^2_d),n}$, we obtain the following corollary.
\begin{cor}\label{cor:multnorm}
For all $f \in \Mult(H^2_d)$, 
\[
\|f\|_{\Mult(H^2_d)} = \sup \{ \|f(T)\| : n \in \bN \,\, and \,\, T \in \CRCo \text{ diagonalizable} \}.
\]
\end{cor}

The fact that 
\[
\|f\|_{\Mult(H^2_d)} = \sup \{ \|f(T)\| : n \in \bN \,\, and \,\, T \in \CRCo \}
\]
has been already observed in the literature on nc functions; see, e.g., \cite[Remark 11.3]{SalShaSha18}.
This raises the question whether 
\[
\|f\|_{\Mult(H^2_d),n} = \sup \{ \|f(T)\| : T \in \CRCo\} ?
\]
We have not been able to answer this question.

We next use Proposition \ref{prop:diagonalizable} to show that the two point norm on the Drury-Arveson space is simply the supremum norm.

\begin{lem}
  \label{lem:two_point}
  Let $F \subset \mathbb{B}_d$ with $|F| \le 2$ and let $m,n \ge 1$.
  If $f \in M_{n,m}(H^\infty(\mathbb{B}_d))$, then there exists
  $g \in M_{n,m}(\Mult(H^2_d))$ with $f \big|_F = g \big|_F$ and $\|g\|_{\Mult(H^2_d)} \le \|f\|_\infty$.
  In particular, if $n=m=1$, then $\|f\|_{\Mult(H^2_d),2} = \|f\|_\infty$.
\end{lem}

\begin{proof}
  If $|F| \le 1$, then we can choose $g$ to be a constant function, so let $|F| = 2$.
  Suppose initially that $F = \{0, r e_1\}$ for some $r \in (0,1)$.
  Let 
  \begin{equation*}
    i: \mathbb{D} \to \mathbb{B}_d, \quad z \mapsto (z,0,\ldots,0),
  \end{equation*}
  be the inclusion and let $P: \mathbb{B}_d \to \mathbb{D}$ be the projection onto the first coordinate.
  Then $f \circ i \in M_{n,m}(H^\infty(\mathbb{D}))$ has norm at most $\|f\|_\infty$.
  The map $h \mapsto h \circ P$ is a complete isometry from $H^\infty(\mathbb{D})$ into $\Mult(H^2_d)$
  (see, for example, \cite[Lemma 6.2]{AHM+20a}).
  Hence, if we define $g = (f \circ i) \circ P$, then
  $g \in M_{n,m}(\Mult(H^2_d))$
  satisfies the conclusions of the lemma.

  If $F$ is an arbitrary two point subset of $\mathbb{B}_d$, then there exists a biholomorphic automorphism $\theta$
  of $\mathbb{B}_d$ so that $\theta(F)$ has the form considered in the first paragraph; see \cite[Section 2.2]{Rudin08}.
  So the result follows from completely isometric automorphism invariance of $H^\infty(\mathbb{B}_d)$ and $\Mult(H^2_d)$; see, e.g.\ Propositions 4.1 and 4.3 in \cite{Hartz17a}.

  As for the additional statement, notice that the inequality $\|f\|_\infty \le \|f\|_{\Mult(H^2_d),2}$ always holds.
  Conversely, if $|F| = 2$, then we apply the first statement to find that
  \begin{equation*}
    \|f \big|_F\|_{\Mult(H^2_d |_F)} = \|g \big|_F\|_{\Mult(H^2_d |_F)} \le \|g\|_{\Mult(H^2_d)}
    \le \|f\|_\infty.
  \end{equation*}
  Taking the supremum over all $F$ yields $\|f\|_{\Mult(H^2_d),2} \le \|f\|_\infty$.
\end{proof}

As a consequence, we obtain a von Neumann-type inequality with constant $1$ for $2 \times 2$ row contractions.

\begin{cor}\label{cor:two}
  \label{cor:2_by_2}
  If $T$ is a commuting $2 \times 2$ row contraction, then
  \begin{equation*}
    \|p(T) \| \le \|p\|_\infty
  \end{equation*}
  for all $p \in \mathbb{C}[z_1,\ldots,z_d]$.
  In other words, we may choose $C_{2} =1$ in Theorem \ref{thm:main}.
\end{cor}

\begin{proof}
  Suppose initially that $T$ is jointly diagonalizable. Applying Proposition \ref{prop:n_point_diagonalizable}
  and Lemma \ref{lem:two_point}, we find that
  \begin{equation}
    \label{eqn:diagonalizable}
    \|p(T)\| \le \|p\|_{\Mult(H^2_d),2} = \|p\|_\infty.
  \end{equation}

  In general, it is known that any tuple $T$ of commuting $2 \times 2$ matrices can be approximated by a sequence $(T_n)$ of commuting diagonalizable
  $2 \times 2$ matrices (see the remarks on page 133 of \cite{HO01}). The row norm of $(T_n)$ converges to the row norm of $T$,
  so by applying \eqref{eqn:diagonalizable} to $r_n T_n$ for a suitable sequence $r_n \in (0,1)$ tending to $1$, the general result follows.
\end{proof}

The following result shows that the last corollary does not extend to $3 \times 3$ matrices.

\begin{prop}
  \label{prop:3_2}
 
There exists a polynomial $p$ so that $\|p\|_{\Mult(H^2_2),3} > \|p\|_\infty$. 
In particular, there exists a pair of commuting $3 \times 3$ matrices that is a row contraction such that $\|p(T)\| > \|p\|_\infty$.
Consequently, $C_{d,n} > 1$ for all $d$ and $n$ such that $d \geq 2$ and $n \geq 3$. 
\end{prop}

\begin{proof}
  Let $p(z) = z_1^2 + z_2^2$, so that $\|p\|_\infty = 1$.
  One can check that the Pick matrix in $H^2_2$
  of $p$ at the points $\left(\frac{4}{5},\frac{1}{5}\right), \left(\frac{1}{5},\frac{4}{5}\right), \left(\frac{2}{5}, \frac{2}{5}\right)$ is not positive semidefinite
  (the determinant is strictly negative). Hence $\|p\|_{\Mult(H^2_2),3} > 1$. 
  The rest follows from Proposition \ref{prop:n_point_diagonalizable} and from the definitions.
\end{proof}

One possible approach to showing that $C_{d,3} < \infty$, extending the basic idea behind the proof of Lemma \ref{lem:two_point}, is to use the special structure of solutions to extremal $3$-point Pick problems on the ball
obtained by Kosi\'{n}sksi and Zwonek \cite{KZ18}.
It is conceivable that the numerical value of the constant $C_{d,3}$ could be determined in this way.
In the next section, we use a somewhat different method, which very likely does not give optimal constants,
but will yield that $C_{d,n} < \infty$ for any $d,n \ge 1$.

\section{von Neumann's inequality up to a constant}
\label{sec:VNI}

Our next goal is to prove Theorem \ref{thm:main} in general.
To this end, we use a variant of the Schur algorithm, somewhat similar to the proof of the main result in \cite{Hartz20}.

We require the solution of Gleason's problem in $H^\infty(\mathbb{B}_d)$, which we state as a lemma for easier reference.
See \cite[Section 6.6]{Rudin08} for a proof.

\begin{lem}
  \label{lem:Gleason}
  Let $d \in \mathbb{N}$.
  There exists a constant $C(d) > 0$ so that for every $f \in H^\infty(\mathbb{B}_d)$ with $f(0) = 0$, there exist $f_1,\ldots,f_d \in H^\infty(\mathbb{B}_d)$
  with $f = \sum_{i=1}^d z_i f_i$ and $\|f_i\|_\infty \le C(d) \|f\|_\infty$ for all $i$.
\end{lem}

Our arguments do not just apply to the Drury--Arveson space, but
to standard weighted spaces on the ball.
For $a > 0$, let $\mathcal{D}_a(\mathbb{B}_d)$ be the RKHS on $\mathbb{B}_d$ with reproducing kernel
\begin{equation*}
  \frac{1}{(1 - \langle z,w \rangle)^a}.
\end{equation*}

\begin{lem}
  \label{lem:norm_a}
  Let $a > 0$. The coordinate functions $z_i$ are multipliers of $\mathcal{D}_a(\mathbb{B}_d)$ and
  \begin{equation*}
    \|
    \begin{bmatrix}
      z_1 & \cdots & z_d
    \end{bmatrix}\|_{\Mult(\mathcal{D}(\mathbb{B}_d) \otimes \mathbb{C}^d, \mathcal{D}_a(\mathbb{B}_d))} =
    \max(1, a^{-1/2}).
  \end{equation*}
\end{lem}

\begin{proof}
  Let $c > 0$. We have to show that the row
  \begin{equation*}
    \begin{bmatrix}
      c z_1 & \ldots c z_d
    \end{bmatrix}
  \end{equation*}
  is a contractive multiplier of $\mathcal{D}_a(\mathbb{B}_d)$ if and only if $c^2 \le \min(1,a)$.
  To this end, a standard result about multipliers (see, e.g.\ \cite[Theorem 6.28]{PR16}) shows that the row is a contractive multiplier if and only if
  the Hermitian kernel $L$ defined by
  \begin{equation*}
    L(z,w) = \frac{1 - c^2 \langle z ,w \rangle}{ (1- \langle z,w \rangle)^a}
  \end{equation*}
  is positive.
  But
  \begin{equation*}
    L(z,w) = (1 - c^2 \langle z,w \rangle ) \sum_{n=0}^\infty (-1)^n \binom{-a}{n} \langle z,w \rangle^n,
  \end{equation*}
  which is positive  if and only if every coefficient of $\langle z,w \rangle^n$ is non-negative
  (see, e.g.\ \cite[Corollary 6.3]{Hartz17a}), which happens if and only if
  \begin{equation*}
    (-1)^{n+1} \binom{-a}{n+1} - c^2 (-1)^n \binom{-a}{n} \ge 0
  \end{equation*}
  for all $n \ge 0$. Since $\binom{-a}{n+1} = \frac{-a-n}{n+1} \binom{-a}{n}$, this happens if and only if
  \begin{equation}
    \label{eqn:norm_ineq}
    c^2 \le \frac{a+n}{n+1} \quad \text{ for all } n \ge 0.
  \end{equation}
  If $0 < a \le 1$, then the function $t \mapsto \frac{a + t}{t+1}$ is increasing, so \eqref{eqn:norm_ineq} holds
  if and only if it holds for $n=0$, that is, if and only if $c^2 \le a$. If $a \ge 1$, then
  the right-hand side of \eqref{eqn:norm_ineq} is at least $1$ and tends to $1$ as $n \to \infty$,
  so \eqref{eqn:norm_ineq} holds if and only if $c^2 \le 1$.
\end{proof}

\begin{rem}
  For $a \ge 1$, Lemma \ref{lem:norm_a} can be easily deduced from the fact that the coordinate
  functions form a row contraction on $H^2_d$. Indeed, since the kernel of $H^2_d$
  is a factor of the kernel of $\mathcal{D}_a(\mathbb{B}_d)$ for $a \ge 1$, the Schur product theorem easily implies
  that the row of the coordinate functions has multiplier norm at most $1$ on  $\mathcal{D}_a(\mathbb{B}_d)$ for $a \ge 1$.
  On the other hand, the row of the coordinate functions has supremum norm $1$, so the multiplier
  norm has to be equal to $1$.
\end{rem}

Unfortunately, it is in general not true that commuting diagonalizable matrices are dense in the set of commuting matrices,
so consideration of the $n$-point norm alone is not sufficient.
Instead, we will work directly with the $n \times n$ matrices.
The following result is the key lemma in the proof of Theorem \ref{thm:main} and some of the later results.

\begin{lem}
  \label{lem:Schur}
  Let $f \in H^\infty(\mathbb{B}_d)$, let $T$ be a $d$-tuple of commuting $n \times n$ matrices whose joint
  spectrum is contained in $\mathbb{B}_d$ and let $a > 0$.
  Then there exists a function $g \in \Mult(\mathcal{D}_a(\mathbb{B}_d))$ with $f(T) = g(T)$ and
  \begin{equation*}
    \|g\|_{\Mult(\mathcal{D}_a(\mathbb{B}_d))} \le (2 C(d) \sqrt{d} \max(1,a^{-1/2}))^{n-1} \|f\|_\infty,
  \end{equation*}
  where $C(d)$ is the constant of Lemma \ref{lem:Gleason}.
\end{lem}

\begin{proof}
  The proof is by induction on $n$. If $n=1$, we may choose $g$ to be a constant function.
  Suppose that $n \ge 2$ and that the statement has been shown for $(n-1) \times (n-1)$ row contractions.

  Let $T$ be a tuple of commuting $n \times n$ matrices whose spectrum is contained in $\mathbb{B}_d$.
  By a unitary change of basis, we may assume that each $T_i$ is upper triangular, say
  \begin{equation*}
    T_i =
    \begin{bmatrix}
      a_i & b_i \\
      0 & A_i
    \end{bmatrix},
  \end{equation*}
  where $a_i$ is a scalar, $b_i$ is a row of length $n-1$, and $A_i$ is an $(n-1) \times (n-1)$ matrix.
  The assumption on the spectrum of $T$ implies that $(a_1,\ldots,a_d) \in \mathbb{B}_d$, so there exists a biholomorphic automorphism $\theta$
  of $\mathbb{B}_d$ that maps $(a_1,\ldots,a_d)$ to $0$.
  Note that the spectrum of $\theta(T)$ is again contained in $\mathbb{B}_d$.
  Thus, replacing $T$ with $\theta(T)$ and $f$ with $f \circ \theta^{-1}$ and using automorphism invariance of $H^\infty(\mathbb{B}_d)$ and
  $\Mult(\mathcal{D}_a(\mathbb{B}_d))$ (see, e.g.\ Propositions 4.1 and 4.3 in \cite{Hartz17a}),
  we may assume that $a_i = 0$ for all $i$; this uses the superposition principle for the functional calculus (Proposition \ref{prop:superposition}). Hence,
  \begin{equation}
    \label{eqn:upper_triangular}
    T_i =
    \begin{bmatrix}
      0 & b_i \\
      0 & A_i
    \end{bmatrix}.
  \end{equation}
  Note that the joint spectrum of the tuple $A = (A_1, \ldots, A_d)$ is also contained in $\bB_d$. 

  Next, let
  \begin{equation*}
    \varepsilon = (2 C(d) \sqrt{d} \max(1,a^{-1/2}))^{-(n-1)},
  \end{equation*}
  and suppose that $\|f\|_\infty \le \varepsilon$.
  We will show that there exists $g \in \Mult(\mathcal{D}_a(\mathbb{B}_d))$ with $f(T) = g(T)$ and $\|g\|_{\Mult(\mathcal{D}_a(\mathbb{B}_d))} \le 1$.
  Let $c = f(0)$ and let $\psi$ be an automorphism
  of $\mathbb{D}$ that maps $c$ to $0$ and $0$ to $c$.
  Define $h = \psi \circ f$.
  Using a standard estimate for holomorphic self-maps of $\mathbb{D}$ \cite[Corollary 2.40]{CM95}, we see that
  \begin{equation*}
    |h(z)| = |\psi(f(z))| \le \frac{|c| + |f(z)|}{1 + |c| |f(z)|} \le 2 \varepsilon.
  \end{equation*}
  Thus, $h \in H^\infty(\mathbb{B}_d)$ with $\|h\|_\infty \le 2 \varepsilon$
  and $h(0) = 0$.
  By Lemma \ref{lem:Gleason}, there exist $h_1,\ldots,h_d \in H^\infty(\mathbb{B}_d)$ with
  $h = \sum_{i=1}^d z_i h_i$ and $\|h_i\|_\infty \le 2 C(d) \varepsilon$.
  From \eqref{eqn:upper_triangular}, we infer that
  \begin{equation}
    \label{eqn:matrix}
    h(T) = \sum_{i=1}^d T_i h_i(T)
    = \sum_{i=1}^d
    \begin{bmatrix}
      0 & b_i \\ 0 & A_i
    \end{bmatrix}
    \begin{bmatrix}
      h_i(0) & * \\
      0 & h_i(A)
    \end{bmatrix}
    = \sum_{i=1}^d
    \begin{bmatrix}
      0 & b_i h_i(A) \\
      0 & A_i h_i(A)
    \end{bmatrix}.
  \end{equation}
  By the inductive hypothesis, there exist $u_1,\ldots,u_d \in \Mult(\mathcal{D}_a(\mathbb{B}_d))$ with $u_i(A) = h_i(A)$ and
  \begin{equation}
    \begin{split}
    \label{eqn:u_i_mult_norm_estimate}
    \|u_i\|_{\Mult(\mathcal{D}_a(\mathbb{B}_d))} &\le (2 C(d) \sqrt{d} \max(a,a^{-1/2}))^{n-2} \|h_i\|_\infty \\ &\le (2 C(d))^{n-1} ( \sqrt{d} \max(1,a^{-1/2}))^{n-2} \varepsilon = (\sqrt{d} \max (1,a^{-1/2}))^{-1}
  \end{split}
  \end{equation}
  by choice of $\varepsilon$.
  Let
  \begin{equation*}
    u = \sum_{i=1}^d z_i u_i =
    \begin{bmatrix}
      z_1 & \cdots & z_d
    \end{bmatrix}
    \begin{bmatrix}
      u_1 \\ \vdots \\ u_d
    \end{bmatrix}.
  \end{equation*}
  From \eqref{eqn:u_i_mult_norm_estimate}, it follows that the column has multiplier norm at most $\max(1,a^{-1/2})^{-1}$, so Lemma \ref{lem:norm_a} implies that
  $u \in \Mult(\mathcal{D}_a(\mathbb{B}_d))$ with $\|u\|_{\Mult(\mathcal{D}_a(\mathbb{B}_d))} \le 1$.
  Moreover, since $u_i(A) = h_i(A)$, the computation in \eqref{eqn:matrix} shows that $u(T) = h(T)$.
  Since $\|u\|_{\Mult(\mathcal{D}_a(\mathbb{B}_d))} \le 1$, we may define $g = \psi^{-1} \circ u$, so that $g \in \Mult(\mathcal{D}_a(\mathbb{B}_d))$ with $\|g\|_{\Mult(\mathcal{D}_a(\mathbb{B}_d))} \le 1$ by the classical von Neumann inequality.
  Moreover, since $h = \psi \circ f$ and $h(T) = u(T)$, it follows that $f(T) = \psi^{-1}(u(T)) = g(T)$.
  This completes the induction and hence the proof.
\end{proof}

The following lemma is useful for improving estimates if
the number of variables $d$ is significantly larger than the size of the matrix $n$.
It is inspired by the result in \cite{KZ18} that solutions to extremal $3$-point Pick problems in any number of variables
only depend on two variables up to automorphisms.

\begin{lem}
  \label{lem:variable_reduction}
  Let $T$ be a $d$-tuple of commuting $n \times n$ matrices whose joint spectrum is contained in $\mathbb{B}_d$.
  \begin{enumerate}[label=\normalfont{(\alph*)}]
    \item If $d \ge n \ge 2$ and if $T$ is jointly diagonalizable, then there exists a biholomorphic automorphism
      $\theta$ of $\mathbb{B}_d$ such that at most the first $n-1$ operators in the $d$-tuple $\theta(T)$
      are non-zero.
    \item If $d > \lfloor {n^2/4} \rfloor + 1$, then there exists a biholomorphic automorphism $\theta$ of $\mathbb{B}_d$,
      given by a $d \times d$ unitary,
      such that at most the first $\lfloor {n^2/4} \rfloor + 1$ operators in the $d$-tuple $\theta(T)$ are non-zero.
  \end{enumerate}
\end{lem}

\begin{proof}
  (a) The joint spectrum $\sigma(T)$ of $T$ consists of at most $n$ points in $\mathbb{B}_d$.
  Thus, we may find a biholomorphic automorphism $\theta$ of $\mathbb{B}_d$ such that $\theta(\sigma(T)) \subset \mathbb{B}_{n-1} \times \{0\}$ by first moving one of the points in $\sigma(T)$ to the origin and then applying a suitable
  $d \times d$ unitary. Since $T$ is jointly diagonalizable, it follows that at most the first $n-1$ entries
  in $\theta(T)$ are non-zero.

  (b) Let $T = (T_1,\ldots,T_d)$ and consider the linear map
  \begin{equation*}
    \Phi: \mathbb{C}^d \to M_n, \quad (\alpha_j)_{j=1}^d \mapsto \sum_{j=1}^d \alpha_j T_j.
  \end{equation*}
  By a theorem of Schur, a commutative subalgebra of $M_n$ has dimension at most $\lfloor n^2/4 \rfloor + 1$.
  It follows that $\ker(\Phi)$ has codimension at most $\lfloor n^2/4 \rfloor + 1$. Therefore, there exists an orthonormal
  basis $(u_j)_{j=1}^d$ of $\mathbb{C}^d$ such that $u_j \in \ker(\Phi)$ for $j >\lfloor {n^2/4} \rfloor + 1$.
  Define a unitary matrix
  \begin{equation*}
    U =
    \begin{bmatrix}
      u_1^T \\ \vdots \\ u_d^T
    \end{bmatrix}
  \end{equation*}
  and let $\theta$ be the automorphism given by $U$.
  By definition, the $j$-th entry of the $d$-tuple $\theta(T)$ is given by $\Phi(u_j)$,
  which is zero if $j >\lfloor {n^2/4} \rfloor + 1$.
\end{proof}

With the help of the preceding lemma, we can improve the constant in Lemma \ref{lem:Schur} in the case
when $d$ is siginificantly larger than $n$.

\begin{lem}
  \label{lem:Schur_improved}
  Let $f \in H^\infty(\mathbb{B}_d)$, let $T$ be a $d$-tuple of commuting $n \times n$ matrices whose joint
  spectrum is contained in $\mathbb{B}_d$ and let $a > 0$.
  Then there exists a function $g \in \Mult(\mathcal{D}_a(\mathbb{B}_d))$ with $f(T) = g(T)$ and
  \begin{equation*}
    \|g\|_{\Mult(\mathcal{D}_a(\mathbb{B}_d))} \le (2 \min(C(d)\sqrt{d}, C(n')\sqrt{n'})  \max(1,a^{-1/2}))^{n-1} \|f\|_\infty,
  \end{equation*}
  where $n' = \lfloor {n^2/4} \rfloor + 1$ and $C(k)$ is the constant of Lemma \ref{lem:Gleason}.
\end{lem}

\begin{proof}
  Let $T = (T_1,\ldots,T_d)$.
  In view of Lemma \ref{lem:Schur}, we may assume that $d > n' = \lfloor {n^2/4} \rfloor + 1$.
  In this case, part (b) of Lemma \ref{lem:variable_reduction} shows that there
  exists a biholomorphic automorphism $\theta$ of $\mathbb{B}_d$ such that 
  at most the first $n'$ entries of $\theta(T)$ are non-zero.
  Thus, by replacing $T$ with $\theta(T)$ and using automorphism invariance
  of $\Mult(\mathcal{D}_a(\mathbb{B}_d))$ as in the proof of Lemma \ref{lem:Schur},
  we may assume that $T_j = 0$ for $j \ge n' +1$.

  Let $P: \mathbb{C}^d \to \mathbb{C}^{n'}$ be the projection onto the first $n'$ coordinates
  and let $i: \mathbb{C}^{n'} \to \mathbb{C}^d$ be the inclusion.
  Let $f \in H^\infty(\mathbb{B}_d)$. Applying Lemma \ref{lem:Schur} to $f \circ i \in H^\infty(\mathbb{B}_{n'})$
  and the shortened tuple $(T_1,\ldots,T_{n'})$, we find $h \in \Mult(\mathcal{D}_a(\mathbb{B}_{n'}))$
  such that
  \begin{equation*}
    h(T_1,\ldots,T_{n'}) = (f \circ i)(T_1,\ldots,T_{n'}) =  f(T)
  \end{equation*}
  and
  \begin{equation*}
    \|h\|_{\Mult(\mathcal{D}_a(\mathbb{B}_{n'}))} \le
    (2 C(n') \sqrt{n'} \max(1,a^{-1/2}))^{n-1} \|f\|_\infty.
  \end{equation*}
  Let $g = h \circ P$. Then $g \in \Mult(\mathcal{D}_a(\mathbb{B}_d))$
  with $\|g\|_{\Mult(\mathcal{D}_a(\mathbb{B}_d))} = \|h\|_{\Mult(\mathcal{D}_a(\mathbb{B}_{n'}))}$,
  see for instance \cite[Lemma 6.2]{AHM+20a}.
  Moreover,
  \begin{equation*}
    g(T) = h(T_1,\ldots,T_{n'}) = f(T),
  \end{equation*}
  as desired.
\end{proof}

The following theorem is a refinement of Theorem \ref{thm:main}.
\begin{thm}\label{thm:vNrow}
  \label{thm:matrix_inequality}
  If $T= (T_1,\ldots,T_d)$ is a commuting $n \times n$ row contraction, then
  \begin{equation*}
    \|p(T)\| \le C_{d,n} \|p\|_\infty 
  \end{equation*}
  for all $p \in \mathbb{C}[z_1,\ldots,z_d]$, where 
  \begin{equation*}
   C_{d,n} \le (2 \min(C(d) \sqrt{d}, C(n') \sqrt{n'}))^{n-1} , 
  \end{equation*}
 $n' = \lfloor {n^2/4} \rfloor + 1$, and $C(k)$ is the constant of Lemma \ref{lem:Gleason}.
\end{thm}

\begin{proof}
As noted in Section \ref{sec:prelim}, we have that $\sigma(T) \subseteq \ol{\bB_d}$. 
  Replacing $T$ with $r T$ for $0 < r < 1$, we may assume that the spectrum of $T$ is contained in $\mathbb{B}_d$.
  By Lemma \ref{lem:Schur_improved}, there exists $g \in \Mult(H^2_d)$ with $g(T) = p(T)$
and $\|g\|_{\Mult(H^2_d)} \le \widetilde{C}_{d,n}$, where $\widetilde{C}_{d,n}= (2 \min(C(d) \sqrt{d}, C(n') \sqrt{n'}))^{n-1} \|p\|_\infty$.
  Hence, by the von Neumann inequality for $H^2_d$ of Drury \cite{Dru78}, M\"uller-Vasilescu \cite{MV93} and Arveson \cite{Arveson98}, we obtain the estimate
  \begin{equation*}
    \|p(T)\| = \|g(T)\| \le \|g\|_{\Mult(H^2_d)} \le \tilde{C}_{d,n} \|p\|_\infty,
  \end{equation*}
  as desired.
\end{proof}

Notice that the constant in Theorem \ref{thm:vNrow} may be bounded above by
\begin{equation*}
  (2 C(n') \sqrt{n'})^{n-1},
\end{equation*}
which is independent of $d$ and only depends on $n$.
However, for fixed $d$ and large $n$, the estimate in Theorem \ref{thm:matrix_inequality}
is better.

\begin{rem}
It follows from the incomparability of the norms $\| \cdot \|_{\Mult(H^2_d)}$ and $\| \cdot \|_{\infty}$ (see the discussion in the introduction) that $C_{d,n} \xrightarrow{n \to \infty} \infty$ for all $d \geq 2$. 
In fact, considering the compression of the tuple $M_z$ on $H^2_d$ to the space of all polynomials of degree at most $k$ and using computations done in the proof of \cite[Theorem 3.3]{Arveson98}, one can show that for $d \ge 2$ and $k \ge 1$,
  \begin{equation*}
    C_{d,n} \ge (2 \pi)^{\frac{d-1}{4}} d^{-\frac{1}{4}} k^{\frac{d-1}{4}},
  \end{equation*}
  where $n = \binom{dk+d}{d}$ is the dimension of the space of polynomials of degree at most $k$ in $d$ variables.
  In particular, for sufficiently large $n$, we have
  \begin{equation*}
    C_{d,n} \geq C_{2,n} \ge n^{\frac{1}{8}}.
  \end{equation*}
  Unfortunately, the gap between these lower bounds and the upper bounds given by Theorem \ref{thm:vNrow} is huge, and we are not able to determine whether there is a strict inequality $C_{2,n} < C_{d,n}$ for any $n > 2$. 
\end{rem}

We also easily obtain a completely bounded version of Theorem \ref{thm:matrix_inequality}.

\begin{cor}\label{cor:vNrow_comp}
  Let $T= (T_1,\ldots,T_d)$ be a commuting $n \times n$ row contraction and let $P \in M_r(\mathbb{C}[z_1,\ldots,z_d])$.
  Then
  \begin{equation*}
    \|P(T)\| \le n C_{d,n} \|P\|_\infty,
  \end{equation*}
  where $C_{d,n}$ is the constant from Theorem \ref{thm:vNrow}.
\end{cor}

\begin{proof}
  Theorem \ref{thm:matrix_inequality} says that the map
  \begin{equation*}
    \mathbb{C}[z_1,\ldots,z_d] \to M_n, \quad p \mapsto p(T),
  \end{equation*}
  is bounded with norm at most $C_{d,n}$ when we regard $\mathbb{C}[z_1,\ldots,z_d]$ as a subspace of $H^\infty(\mathbb{B}_d)$.
  By basic operator space theory (see \cite[Corollary 2.2.4]{ER00}), it follows that the map
  is completely bounded with completely bounded norm at most $n C_{d,n}$.
\end{proof}

%

\section{Gleason's problem}\label{sec:Gleason}

Recall that Gleason's problem for $H^\infty(\mathbb{B}_d)$ is the question of whether every function $f \in H^\infty(\mathbb{B}_d)$
with $f(0) = 0$ can we written as
\begin{equation*}
  f = \sum_{i=1}^d z_i f_i
\end{equation*}
for some $f_1,\ldots,f_d \in H^\infty(\mathbb{B}_d)$.
Work of Leibenson and of Ahern and Schneider shows that this question has a positive answer (see \cite[Section 6.6]{Rudin08} and Lemma \ref{lem:Gleason}).
Thus, it is natural to ask about the minimal possible norm of a tuple $(f_1,\ldots,f_d)$ of solutions.
In \cite{Doubtsov98}, Doubtsov studied linear operators 
\begin{equation*}
  L: \{f \in H^\infty(\mathbb{B}_d): f(0) = 0 \} \to H^\infty(\mathbb{B}_d)^d
\end{equation*}
solving Gleason's problem; he showed that the solution of Leibenson and Ahern--Schneider gives the minimal norm among all such operators.
Moreover, he determined the minimal norm when $H^\infty(\mathbb{B}_d)$ is replaced with $H^2(\mathbb{B}_d)$.

We now use the fact that $C_{2,3} > 1$ to show that Gleason's problem cannot be solved contractively
in $H^\infty(\mathbb{B}_d)$ for $d \ge 2$.
The idea is that if Gleason's problem could be solved contractively, then one could
use the Schur algorithm as in Section \ref{sec:VNI} and the fact that the two point norm on $H^2_d$
equals the supremum norm in the vector-valued setting (Lemma \ref{lem:two_point}) to show that the three point
norm equals the supremum norm as well.

\begin{prop}
  Let $d \ge 2$. There exists a rational function $f \in A(\mathbb{B}_d)$ with $f(0) = 0$
  and $\|f\|_\infty = 1$
  that cannot be written as
  $f = \sum_{i=1}^d z_i f_i$
  with $f_i \in H^\infty(\mathbb{B}_d)$ and $\sup_{z \in \mathbb{B}_d} \sum_{i=1}^d |f_i(z)|^2 \le 1$.
\end{prop}

\begin{proof}
  Clearly, it suffices to consider $d=2$.
  By Proposition \ref{prop:3_2}, there exists a polynomial $p$ and a set $F = \{\lambda_1,\lambda_2,\lambda_3\} \subset \mathbb{B}_2$
  so that 
  \begin{equation}
    \label{eqn:p_big}
    \|p \big|_F\|_{\Mult(H^2_2 \big|_F)} > \|p\|_\infty = 1.
  \end{equation}
  Let $\theta$ be a biholomorphic automorphism of $\mathbb{B}_2$
  mapping $\lambda_1$ to $0$ and let $\psi$ be a biholomorphic automorphism of $\mathbb{D}$ mapping $p(\lambda_1)$ to $0$. Set $f = \psi \circ p \circ \theta^{-1}$. Then $f$ is a rational function in $A(\mathbb{B}_2)$ with $f(0) = 0$
  and $\|f\|_\infty = 1$.
  We claim that Gleason's problem for $f$ cannot be solved contractively.

  Suppose towards a contradiction that there exist $f_1,f_2 \in H^\infty(\mathbb{B}_2)$ so that
  \begin{equation*}
    f =
    \begin{bmatrix}
      z_1 & z_2
    \end{bmatrix}
    \begin{bmatrix}
      f_1 \\ f_2
    \end{bmatrix}
    \quad \text{ and }
    \left\|
    \begin{bmatrix}
      f_1 \\ f_2
    \end{bmatrix}
    \right\|_{\infty} \le 1.
  \end{equation*}
  Write $\theta(F) = \{0\} \cup F'$ with $|F'| = 2$.
  Applying Lemma \ref{lem:two_point} to
  the column
  $\begin{bmatrix}
    f_1 \\ f_2
  \end{bmatrix}$,
  we find a column
  $
  \begin{bmatrix}
    h_1 \\ h_2
  \end{bmatrix}
  \in M_{2,1}(\Mult(H^2_d))$ of norm at most $1$
  so that $f_i \big|_{F'} = h_i \big|_{F'}$ for $i = 1,2$.
  Define
  \begin{equation*}
    g =
    \begin{bmatrix}
      z_1 & z_2
    \end{bmatrix}
    \begin{bmatrix}
      h_1 \\ h_2
    \end{bmatrix}.
  \end{equation*}
  Then $\|g\|_{\Mult(H^2_d)} \le 1$ and $g \big|_{\theta(F)} = f \big|_{\theta(F)}$, because $f$ and $g$
  agree on $F'$ and at $0$.
  Set $u = \psi^{-1} \circ g \circ \theta$. By automorphism invariance of $\Mult(H^2_d)$ and by the classical
  von Neumann inequality, we have $\|u\|_{\Mult(H^2_d)} \le 1$.
  Moreover, since $f = \psi \circ p \circ \theta^{-1}$, we find that $u$ agrees with $p$ on $F$.
  This contradicts \eqref{eqn:p_big} and hence finishes the proof.
\end{proof}

\section{Applications to RKHS on the ball}\label{sec:reps}

In this section, we show that our methods can also be used to answer Question 10.3 of \cite{AHM+20a}.
It was shown in \cite{AHM+20a} that for $0 \le a < \frac{d+1}{2}$,
the $n$-point multiplier norm on $\mathcal{D}_a(\mathbb{B}_d)$ is not
comparable to the full multiplier norm, and in fact $\Mult(\mathcal{D}_a(\mathbb{B}_d))$ is not topologically
subhomogeneous. 
This means that there do not exist constants $c$ and $C$ and an integer $N$ such that
for all $f \in \Mult(\mathcal{D}_a(\mathbb{B}_d))$,
\begin{equation*}
  \|f\|_{\Mult(\mathcal{D}_a(\mathbb{B}_d))} \le c \sup \{ \|\pi(f)\| \},
\end{equation*}
where the supremum is taken over all unital homomorphisms $\pi: \mathcal{A} \to M_k$ with $\|\pi\| \le C$
and all $k \le N$.

If $a \ge d$, then $\Mult(\mathcal{D}_a(\mathbb{B}_d)) = H^\infty(\mathbb{B}_d)$ completely
isometrically, so it was asked whether $\Mult(\mathcal{D}_a(\mathbb{B}_d))$ is topologically subhomogeneous
for $\frac{d+1}{2} \le a < d$.

To answer this question, we require the following consequence of Lemma \ref{lem:Schur_improved}.
We let $A(\mathcal{D}_a(\mathbb{B}_d))$ denote the norm closure of the polynomials
in $\Mult(\mathcal{D}_a(\mathbb{B}_d))$.

\begin{lem}
  \label{lem:DA_hom}
  Let $a > 0$. There exists a constant $C(a,n)$ so that for any unital bounded
  homomorphism $\pi: A(\mathcal{D}_a(\mathbb{B}_d)) \to M_n$, we have
  \begin{equation*}
    \|\pi(f)\| \le C(a,n) \|\pi\| \|f\|_\infty
  \end{equation*}
  for all $f \in A(\mathcal{D}_a(\mathbb{B}_d))$.
\end{lem}

\begin{proof}
  Let $\pi: A(\mathcal{D}_a(\mathbb{B}_d)) \to M_n$ be a unital bounded homomorphism and let
  $T = (T_1,\ldots,T_d) = (\pi(M_{z_1}), \ldots, \pi(M_{z_d}))$,
  so that $\pi(p) = p(T)$ for all polynomials $p$.
  Recall that the spectrum of $M_z$ in the Banach algebra $A(\mathcal{D}_a(\mathbb{B}_d))$ is contained
in $\overline{\mathbb{B}}_d$
  (for $a \le 1$, this follows for instance from the discussion following Lemma 5.3 in \cite{CH18};
  for $a \ge 1$, it holds since $M_z$ is a row contraction).
Hence the joint spectrum of $T$ is also contained in $\overline{\mathbb{B}}_d$.

  If $f \in \mathcal{O}(\mathbb{B}_d)$ and $0 <r < 1$, let $f_r(z) = f(r z)$.
  It is well known that the map $f \mapsto f_r$ is a unital completely contractive
  homomorphism from $\Mult(\mathcal{D}_a(\mathbb{B}_d))$ into $A(\mathcal{D}_a(\mathbb{B}_d))$. 
  Defining $\pi_r(f) = \pi(f_r)$,
  we obtain a unital homomorphism $\pi_r: \Mult(\mathcal{D}_a(\mathbb{B}_d)) \to M_n$ with
  $\|\pi_r\| \le \|\pi\|$ and $\pi_r(M_z) = r T$.
  Given a polyonomial $p$ and $0 <r < 1$,
  Lemma \ref{lem:Schur_improved} yields a constant $C(a,n)$ and a function
  $g \in \Mult(\mathcal{D}_a(\mathbb{B}_d))$ so that $g(r T) = p( r T)$ and $\|g\|_{\Mult(\mathcal{D}_a(\mathbb{B}_d))} \le C(a,n) \|p\|_\infty$.
  Moreover, since $g_r \in A(\mathcal{D}_a(\mathbb{B}_d))$ and $\pi$ is continuous,
  \begin{equation*}
    \pi_r(p) = p(r T) = g(r T) = \pi( g_r) = \pi_r(g).
  \end{equation*}
  Thus
  \begin{equation*}
    \|\pi_r(p)\| = \|\pi_r(g)\| \le \|\pi_r\| \|g\|_{\Mult(\mathcal{D}_a(\mathbb{B}_d))}
    \le C(a,n) \|\pi\| \|p\|_\infty.
  \end{equation*}
  Taking the limit $r \to 1$, we obtain the desired conclusion for polynomials.
  By continuity, the result follows for all $f \in A(\mathcal{D}_a(\mathbb{D}_d))$.
\end{proof}

We can now answer \cite[Question 10.3]{AHM+20a}.
\begin{thm}\label{thm:not_top_sub}
  Let $0 <a < d$. Then the algebras $A(\mathcal{D}_a(\mathbb{B}_d))$ and $\Mult(\mathcal{D}_a(\mathbb{B}_d))$
  are not topologically subhomogeneous.
\end{thm}

\begin{proof}
  If $0 < a < d$, then the multiplier norm on $\mathcal{D}_a(\mathbb{B}_d)$ is not dominated
  by a constant times the the supremum norm for polynomials; see
  \cite[Proposition 9.7]{AHM+20a} and its proof.
  Thus, it follows from Lemma \ref{lem:DA_hom} that $A(\mathcal{D}_a(\mathbb{B}_d))$ is not
  topologically subhomogeneous. Hence, the larger algebra $\Mult(\mathcal{D}_a(\mathbb{B}_d))$
  is not topologically subhomogeneous either.
\end{proof}

We also see that for $a > 0$, the $n$-point multiplier norm on $\mathcal{D}_a(\mathbb{B}_d)$
is comparable to the supremum norm. If $d=1$, this was shown in 
\cite[Corollary 3.3]{AHM+20a}.

\begin{cor}
  Let $a > 0$. Then there exists a constant $C(a,n)$ so that
  \begin{equation*}
    \|f\|_\infty \le \|f\|_{\Mult(\mathcal{D}_a(\mathbb{B}_d)),n} \le C(a,n)
    \|f\|_\infty
  \end{equation*}
  for all $f \in H^\infty(\mathbb{B}_d)$.
\end{cor}

\begin{proof}
  The first inequality is trivial.
  For the second inequality, let $f \in H^\infty(\mathbb{B}_d)$ and
  suppose that $F \subset \mathbb{B}_d$ with $|F| \le n$.
  Applying Lemma \ref{lem:Schur_improved} to the diagonal tuple $T = (T_1,\ldots,T_d) \in M_n^d$
  whose entries are the points of $F$,
  we obtain a constant $C(a,n)$ and $g \in \Mult(\mathcal{D}_a(\mathbb{B}_d))$ with $g \big|_F = f \big|_F$
  and $\|g\|_{\Mult(\mathcal{D}_a(\mathbb{B}_d))} \le C(a,n) \|f\|_\infty$. Hence
  \begin{equation*}
    \|f \big|_F\|_{\Mult(\mathcal{D}_a(\mathbb{B}_d) \big|_F)} =
    \|g \big|_F\|_{\Mult(\mathcal{D}_a(\mathbb{B}_d) \big|_F)} \le \|g\|_{\Mult(\mathcal{D}_a(\mathbb{B}_d))}
    \le C(a,n) \|f\|_\infty.
  \end{equation*}
  Hence the second inequality holds.
\end{proof}

We now use Theorem \ref{thm:vNrow} and Lemma \ref{lem:DA_hom} to determine the finite dimensional representations of the algebras $A(\mathcal{D}_a(\mathbb{B}_d))$. 
Let $A(\bB_d)$ denote the {\em ball algebra}, that is, the algebra of holomorphic functions in $\bB_d$ which extend continuously to $\ol{\bB}_d$. 
The ball algebra is the closure in $H^\infty(\bB_d)$, with respect to the supremum norm, of the polynomials. 
Recall that if $a \ge d$, then $A(\mathcal{D}_a(\mathbb{B}_d)) = A(\mathbb{B}_d)$ completely
isometrically.

\begin{thm}\label{thm:reps_same}
  For all $a > 0$, the unital bounded $n$-dimensional representations of $A(\mathcal{D}_a(\mathbb{B}_d))$ coincide with those of $A(\bB_d)$. 
  Every such representation is uniquely determined by a $d$-tuple $T$, which is jointly similar to a row contraction, and such that $\pi(p) = p(T)$ for every polynomial $p \in \bC[z_1, \ldots, z_d]$. 
\end{thm}
\begin{proof}
Suppose that $\pi: A(\mathcal{D}_a(\mathbb{B}_d)) \to M_n$ is a bounded unital homomorphism. 
By the first part of the proof of Lemma \ref{lem:DA_hom}, there is a $d$-tuple $T$ with $\sigma(T) \subseteq \ol{\bB}_d$, such that $\pi(p) = p(T)$ for every polynomial. 
By Lemma \ref{lem:DA_hom}, this extends uniquely to a continuous unital representation of $A(\bB_d)$. 
Indeed, when can simply define $f(T) := \pi(f) := \lim\pi(p_n)$ where $p_n$ are polynomials that converge to $f$ uniformly on the closed ball. 

Conversly, if $\pi: A(\bB_d) \to M_n$ is a bounded unital homomorphism, then it restricts to the subalgebra $A(\cD_a(\bB_d))$, and because the multiplier norm is always bigger than the supremum norm, $\pi\big|_{A(\cD_a(\bB_d))}$ is also bounded. 

Finally, since the bounded and unital $n$-dimensional representations of all the algebras $A(\cD_a(\bB_d))$ coincide, and since a representation is clearly determined by the images of the coordinate functions $T_1 = \pi(z_1), \ldots, T_d = \pi(z_d)$, it suffices to identify the representations of $\cA_d := A(\mathcal{D}_1(\mathbb{B}_d))$. 
However, by \cite[Proposition 10.1]{SalShaSha20}, the bounded representations of $\cA_d$ are in one to one correspondence with the $d$-tuples $T$ which are jointly similar to a row contraction.
\end{proof}

\begin{rem}
It is natural to ask whether one may replace the condition for the $d$-tuple $T$ to be similar to a row contraction, with the condition that $\sigma(T) \subseteq \ol{\bB}_d$.
However, this is not true. 
For example, consider the case $d=1$ and $T = \big(\begin{smallmatrix} 1 &1 \\ 0 & 1 \end{smallmatrix} \big)$ . 
Then $\sigma(T) = \{1\} \subseteq \ol{\bD}$, but $T$ is not similar to a contraction. 
Indeed, $T$ does not define a bounded representation of $A(\bD)$, as $T$ is not power bounded.
\end{rem}

\section{An application to uniform continuity of noncommutative functions}\label{sec:application}

In this section we use our previous results to answer an open question regarding uniform continuity of noncommutative (nc) functions on the nc unit ball \cite{SalShaSha18}. 
For a thorough introduction to the theory of nc functions, see \cite{KVV14}; the beginner might prefer to start with the expository paper \cite{AM16}.

The multiplier algebra $\Mult(H^2_d)$ can be identified (via the functional calculus) with the algebra $H^\infty(\mathfrak{C}\mathfrak{B}_d)$ of bounded nc holomorphic functions on the nc variety $\mathfrak{C}\mathfrak{B}_d = \sqcup_{n=1}^\infty \CRCo$, and we have $\|f\|_{\Mult(H^2_d)} = \sup_{T \in \mathfrak{C}\mathfrak{B}_d}\|f(T)\|$; for a discussion of this point of view see \cite{SalShaSha18} (in particular Section 11).

For brevity and to be compatible with other parts of the literature, let us write $\cA_d$ for the norm closure of the polynomials in $\Mult(H^2_d)$, that is $\cA_d$ will be just shorthand for $A(\cD_1(\bB_d))$. 
By \cite[Corollary 9.4]{SalShaSha18}, $\cA_d$ equals the subalgebra of $H^\infty(\mathfrak{C}\mathfrak{B}_d)$ consisting of all bounded nc functions that extend uniformly continuously to $\ol{\mathfrak{C}\mathfrak{B}_d} = \sqcup_{n=1}^\infty \CRC$. 
Here, $f \in H^\infty(\mathfrak{C}\mathfrak{B}_d)$ is said to be {\em uniformly continuous} if for every $\epsilon > 0$, there exists a $\delta > 0$ such that for all $n$ and all $X,Y \in \CRC$, $\|X-Y\| < \delta$ implies $\|f(X) - f(Y)\|<\epsilon$. 

By Theorem \ref{thm:reps_same}, every row contraction $T \in \CRC$ gives rise to a a bounded unital representation of $A(\bB_d)$, which we denote $f \mapsto f(T)$. 

\begin{prop}\label{prop:levevelwise_uc}
For every $f \in A(\bB_d)$ and every $n$, the map
\begin{equation*}
  \CRC \to M_n, \quad T \mapsto f(T),
\end{equation*}
is uniformly continuous, in the sense that for all $n$, and for every $\epsilon > 0$, there exists a $\delta > 0$ such that $X,Y \in \CRC$ and $\|X-Y\| < \delta$ implies $\|f(X) - f(Y)\|<\epsilon$. 
\end{prop}
\begin{proof}
It is clear that every $p \in \bC[z_1, \ldots,z_d]$ can be evaluated at every $T \in \CRC$, and that $p$ is uniformly continuous on $\CRC$. 
Moreover, $A(\mathbb{B}_d)$ is the closure of $\mathbb{C}[z_1,\ldots,z_d]$ with respect to the supremum norm.
Therefore, if $f \in A(\mathbb{B}_d)$ and $p_n \in \bC[z_1, \ldots, z_d]$ is a sequence of polynomials that converges in the supremum norm to $f$, then Theorem \ref{thm:vNrow} implies that $p_n(T)$ converges in norm to $f(T)$, and the convergence is uniform in $T \in \CRC$.
As the uniform limit of the uniformly continuous functions $p_n : \CRC \to M_n$, the function $f: \CRC \to M_n$ is also uniformly continuous. 
\end{proof}

By the proposition, every $f \in A(\bB_d)$ extends to a function on $\mathfrak{C}\mathfrak{B}_d$, and since $f(T)$ is given by the functional calculus (see the appendix) it is not hard to see that $f$ is actually an nc function. 
Moreover, $f$ is {\em levelwise bounded} and also {\em levelwise uniformly continuous}, in the obvious sense. 
However, there are functions in $A(\bB_d)$ which are not multipliers, and hence not uniformly bounded on $\mathfrak{C}\mathfrak{B}_d$ (see Section 3.7 ``The strict containment $\cM_d \subsetneq H^\infty(\bB_d)$'' in \cite{ShalitSurvey}).

Since bounded noncommutative functions have some remarkable regularity properties, it might seem plausible that an nc function that is both globally bounded and uniformly continuous on every level $\CRCo$, will be forced somehow to be uniformly continuous on $\mathfrak{C}\mathfrak{B}_d$. 
Question 9.16 in \cite{SalShaSha18} asked whether there exist functions in $H^\infty(\mathfrak{C}\mathfrak{B})$ that are levelwise uniformly continuous but not uniformly continuous. 
We can now show that the answer to this question is positive.

\begin{prop}\label{prop:lw_uniform}
There exists a function $f \in \Mult(H^2_d) = H^\infty(\mathfrak{C}\mathfrak{B}_d)$ which is levelwise uniformly continuous, but not uniformly continuous on $\mathfrak{C}\mathfrak{B}_d$.
\end{prop}
\begin{proof}
By Proposition \ref{prop:levevelwise_uc} every function in $A(\bB_d)$ is levelwise uniformly continuous. 
However, we know that a bounded nc function is uniformly continuous on $\mathfrak{C}\mathfrak{B}_d$ if and only if it is in $\cA_d$. 
Thus it all boils down to the question whether there exists a function $f \in \Mult(H^2_d) \cap A(\bB_d) = \Mult(H^2_d) \cap C(\ol{\bB}_d)$ which is not in $\cA_d$. 
The existence of such a function was established in Section 5.2 of \cite{ShalitSurvey} (``Continuous multipliers versus $\cA_d$''), where it was explained how this follows from the methods of \cite{FangXia11}. 
\end{proof}

\appendix
\section{An elementary construction of a holomorphic functional calculus for balls}

The following result is a special case of the Arens--Calderon functional calculus.
If $B \subset \mathbb{C}^d$ is an open ball, we equip $\mathcal{O}(B)$ with the topology of of uniform convergence
on compact subsets of $B$.

\begin{thm}
  \label{thm:fc_elementary}
  Let $\mathfrak{B}$ be a commutative unital Banach algebra, let $a = (a_1,\ldots,a_d) \in \mathfrak{B}^d$ and let $B$ be an open ball
  containing the joint spectrum $\sigma_{\mathfrak{B}}(a)$.
  Then there exists a unique continuous homomorphism
  $\Phi: \mathcal{O}(B) \to \mathfrak{B}$ such that $\Phi(p) = p(a)$ for every polynomial $p \in \mathbb{C}[z_1,\ldots,z_d]$.
\end{thm}

\begin{proof}
  Uniqueness follows from density of the polynomials in $\mathcal{O}(B)$, as $B$ is a ball.
  To show existence, we define $\Phi$ by adapting the Cauchy integral formula for balls, see \cite[Section 3.2]{Rudin08}.
  Applying a translation and replacing $a$ with $r a$ for a suitable number $r \in (0,\infty)$, it suffices to show that
  if $\sigma(a) \subset \mathbb{B}_d$ and $R > 1$, then there exists a continuous homomorphism
  $\Phi: \mathcal{O}(B_R(0)) \to \mathfrak{B}$ extending the polynomial functional calculus.

  Given $\zeta = (\zeta_1,\ldots,\zeta_d) \in \partial \mathbb{B}_d$, define
  \begin{equation*}
    \langle a,\zeta \rangle  = \sum_{i=1}^d a_i \overline{\zeta_i}.
  \end{equation*}
  Since $\sigma(a) \subset \mathbb{B}_d$ is compact, the Cauchy-Schwarz inequality in $\mathbb{C}^d$ implies that $\sigma( \langle a,\zeta \rangle ) \subset \{\lambda\in \mathbb C: |\lambda|\le s\}$ for some $0<s<1$ 
  for all $\zeta \in \partial \mathbb{B}_d$. Using continuity of the inverse in $\mathfrak{B}$, we may therefore define
  \begin{equation*}
    \Phi(f) = \int_{\partial \mathbb{B}_d} f(\zeta) (1 - \langle a, \zeta \rangle)^{-d} d \sigma(\zeta) \quad (f \in \mathcal{O}(B_R(0)),
  \end{equation*}
  where $d \sigma$ is the normalized surface measure on $\partial \mathbb{B}_d$.
  It is clear that $\Phi$ is linear. Moreover,
  \begin{equation*}
    \|\Phi(f)\| \le \sup_{\zeta \in \partial \mathbb{B}_d} |f(\zeta)| \sup_{\zeta \in \partial \mathbb{B}_d} \| ( 1 - \langle a, \zeta \rangle)^{-d}\|,
  \end{equation*}
  where the last factor is finite by compactness of $\partial \mathbb{B}_d$ and continuity of the inverse in $\mathfrak{B}$.
  Thus, $\Phi$ is continuous.

  We finish the proof by showing that $\Phi(z^\beta) = a^\beta$ for every monomial $z^\beta$. Since the polynomial functional calculus
  is a homomorphism, it then follows that $\Phi$ is a homomorphism as well.
  As $\sigma( \langle a,\zeta \rangle) \subset \{\lambda \in \mathbb C:|\lambda|\le s\} $ for all $\zeta \in \partial \mathbb{B}_d$,
  the spectrum of $\zeta \mapsto \langle a ,\zeta \rangle$ in the Banach algebra $C(\partial \mathbb{B}_d,\mathfrak{B})$
  of all
  continuous functions from $\partial \mathbb{B}_d$ into $\mathfrak{B}$ is contained in $\{\lambda \in \mathbb C:|\lambda|\le s\}$ as well.
  Applying the spectral radius formula in $C(\partial \mathbb{B}_d,\mathfrak{B})$,
  we find that $\lim_{n \to \infty} \sup_{\zeta \in \partial \mathbb{B}_d} \| \langle a,\zeta \rangle^n\|^{1/n} \le s < 1$.
  Consequently, we may expand $(1 - \langle a,\zeta \rangle)^{-d}$ into a binomial series that
  converges uniformly in $\zeta \in \partial \mathbb{B}_d$, so
  \begin{equation*}
    \Phi(z^\beta)
    = \sum_{n=0}^\infty \binom{d+n-1}{n} \sum_{|\alpha| = n} \binom{n}{\alpha} a^\alpha \int_{\partial \mathbb{B}_d} \zeta^\beta \overline{\zeta}^\alpha d \sigma(\zeta).
  \end{equation*}
  A basic orthogonality relation
  for the surface integral (Propositions 1.4.8 and 1.4.9 in \cite{Rudin08}) shows that
  \begin{equation*}
    \int_{\partial \mathbb{B}_d} \zeta^\beta \overline{\zeta}^\alpha d \sigma(\zeta) = \delta_{\alpha \beta} \frac{(d-1)! \alpha!}{(d-1 + |\alpha|)!},
  \end{equation*}
  so $\Phi(z^\beta) = a^\beta$, as desired.
\end{proof}

Alternatively, in the case of a tuple $T= (T_1,\ldots,T_d)$ of commuting Hilbert space operators,
it is possible to define the holomorphic functional calculus $\Phi$ of Theorem \ref{thm:fc_elementary} with
the help of convergent power series. This uses one inequality of the multivariable spectral radius formula \cite[Theorem 1]{MS92}, namely
\begin{equation*}
  \limsup_{n \to \infty} \Big\| \sum_{|\alpha| = n} \binom{n}{\alpha} (T^*)^\alpha T^\alpha \Big\|^{1/2n} \le \sup \{ |\lambda| : \lambda \in \sigma(T) \},
\end{equation*}
which can be proved in an elementary fashion. Using this, one shows that if $\sigma(T) \subset \mathbb{B}_d$,
then for each $f \in H^2_d$ with homogeneous expansion $f = \sum_{n=0}^\infty f_n$, the series
$\sum_{n=0}^\infty f_n(T)$ converges absolutely, from which the holomorphic functional calculus
can be easily deduced. We omit the details.

As is customary, we usually write $f(a)$ for $\Phi(f)$ in the setting of Theorem \ref{thm:fc_elementary}.
The superposition principle for the functional calculus in our particular setting can also be obtained
by elementary means.

\begin{prop}
  \label{prop:superposition}
  Let $\mathfrak{B}$ be a commutative unital Banach algebra, let $a \in \mathfrak{B}^d$
  and let $B$ be an open ball containing the joint spectrum $\sigma_{\mathfrak{B}}(a)$.
  Let $f = (f_1,\ldots,f_k) \in \mathcal{O}(B)^k$ and write $f(a) = (f_1(a),\ldots,f_k(a))$.
  Then:
  \begin{enumerate}[label=\normalfont{(\alph*)}]
    \item $\sigma_{\mathfrak{B}}(f(a)) = f(\sigma_{\mathfrak{B}}(a))$.
    \item If $g$ is a holomorphic function on an open ball containing $f(B)$, then $g(f(a)) = (g \circ f)(a)$.
  \end{enumerate}
\end{prop}

\begin{proof}
  (a) Let $a = (a_1,\ldots,a_d)$. If $\chi$ is a character on $\mathfrak{B}$, then
  for every polnomial $p$, we have
  $\chi(p(a)) = p( \chi(a_1),\ldots,\chi(a_d))$,
  hence $\chi(f_j(a)) = f_j(\chi(a_1), \ldots, \chi(a_d))$ for $j = 1,\ldots,d$
  by an approximation argument. The definition of joint spectrum then yields $\sigma_{\mathfrak{B}}(f(a)) = f(\sigma_{\mathfrak{B}}(a))$.

  (b) Since the functional calculus for the tuple $a$ is a homomorphism,
  we have $p(f(a)) = (p \circ f)(a)$ for every polynomial $p$.
  The general case then follows from an approximation argument.
\end{proof}

\bibliographystyle{plain}
\bibliography{literature}
\end{document}